\documentclass[11pt]{amsart}
\usepackage[a4paper, total={6in, 8in}]{geometry}
\usepackage{enumerate}
\usepackage{amsmath}
\usepackage{amssymb,latexsym}
\usepackage{amsthm}
\usepackage{color}
\usepackage{cancel}
\usepackage{graphicx}
\usepackage{cite}
\usepackage{changes}
\usepackage{hyperref}
\usepackage{comment}
\usepackage{amscd}
\usepackage{mathtools}

\DeclareMathOperator{\C}{\mathcal{C}}

\usepackage[foot]{amsaddr}
\DeclareMathOperator{\GammaL}{\Gamma\mathrm{L}}

\newtheorem{theorem}{Theorem}[section]

\newtheorem{corollary}[theorem]{Corollary}

\newtheorem{proposition}[theorem]{Proposition}

\newtheorem{remark}[theorem]{Remark}

\newcommand{\fqn}{\mathbb{F}_{q^n}}

\newcommand{\F}{{\mathbb F}}

\newcommand{\fq}{{\mathbb F}_{q}}

\newcommand{\la}{\langle}
\newcommand{\ra}{\rangle}

\newcommand{\PG}{\mathrm{PG}}
\newcommand{\N}{\mathrm{N}}

\newcommand{\classcode}{\mathcal{C}[m,r,d]_q}
\newcommand{\classsystem}{\mathcal{P}[m,r,d]_q}

\title{Two pointsets in $\mathrm{PG}(2,q^n)$ and the associated codes}
\date{}

\author[V. Napolitano]{Vito Napolitano}
\address{Vito Napolitano, \textnormal{Dipartimento di Matematica e Fisica, Universit\`a degli Studi della Campania ``Luigi Vanvitelli'', Viale Lincoln, 5, I--\,81100 Caserta, Italy}}
\email{vito.napolitano@unicampania.it}

\author[O. Polverino]{Olga Polverino}
\address{Olga Polverino, \textnormal{Dipartimento di Matematica e Fisica, Universit\`a degli Studi della Campania ``Luigi Vanvitelli'', Viale Lincoln, 5, I--\,81100 Caserta, Italy}}
\email{olga.polverino@unicampania.it}

\author[P. Santonastaso]{Paolo Santonastaso}
\address{Paolo Santonastaso, \textnormal{Dipartimento di Matematica e Fisica, Universit\`a degli Studi della Campania ``Luigi Vanvitelli'', Viale Lincoln, 5, I--\,81100 Caserta, Italy}}
\email{paolo.santonastaso@unicampania.it}

\author[F. Zullo]{Ferdinando Zullo}
\address{Ferdinando Zullo, \textnormal{Dipartimento di Matematica e Fisica, Universit\`a degli Studi della Campania ``Luigi Vanvitelli'', Viale Lincoln, 5, I--\,81100 Caserta, Italy}}
\email{ferdinando.zullo@unicampania.it}

\begin{document}

\maketitle

\begin{abstract}
In this paper we consider two pointsets in $\mathrm{PG}(2,q^n)$ arising from a linear set $L$ of rank $n$ contained in a line of $\mathrm{PG}(2,q^n)$: the first one is a linear blocking set of R\'edei type, the second one extends the construction of translation KM-arcs.
We  point out that their intersections pattern with lines is related to the weight distribution of the considered linear set $L$. We then consider the Hamming metric codes associated with both these constructions, for which we can completely describe their weight distributions. 
By choosing $L$ to be an $\fq$-linear set with a {\it short} weight distribution, then the associated codes have \emph{few weights}. We conclude the paper by providing a connection between the $\Gamma\mathrm{L}$-class of $L$ and the number of inequivalent codes we can construct starting from it.
\end{abstract}

\section{Introduction}
In the study of the geometry of $\mathrm{PG}(r, q)$,   special attention has been paid to the study of its subsets of points based on their behavior with respect to the intersections they have with all the members of one or more  assigned  families of subspaces. Such sets, when  the prescribed  family of the subspaces is that of the hyperplanes and the range of the size of their intersections with hyperplanes is small,  are related with other combinatorial objects such as for example graphs,  some classes of difference sets and linear codes with few weights, see e.g. \cite{CG1984,CK1986,DingDing,Durante,ShiSo}. When considering this last relationship, to which we are interested in this paper,  the existence problem  takes on a special role. Indeed, such sets gives rise to a linear code with few weights whose generator matrix is given by a matrix whose columns are the coordinates of the  points of the given set and vice versa, given  a linear code with few weights, the columns of (one of)  its generator matrix seen as coordinates of points in a  projective space form  a (multi)set   with few intersection sizes with respect to the hyperplanes of the space.

Our aim is to find codes arising from linear sets (or pointsets related to them) having \emph{few} weights in the spirit of \cite{NZ,ZiniZullo}, see also \cite[Section 4]{Alfarano2}. Let $\PG(2,q^n)=\PG(V,\F_{q^n})$, with $\dim_{\F_{q^n}}(V)=3$, up to coordinatize, we can suppose that $V=\F_{q^n}^3$.
In this paper, we exploit two constructions of pointsets in $\PG(2,q^n)$ arising from linear sets of rank $n$ on a projective line. Let $L_U$ be a linear set of rank $n$ in $\mathrm{PG}(1,q^n)$ (up to equivalence) $L_U=\{\langle (x,f(x) \rangle \colon x \in \F_{q^n}^*\} \subseteq \PG(1,q^n)$, where $f(x)$ is a $q$-polynomial. Then we may define the following two pointsets of $\PG(2,q^n)$:
\begin{enumerate}
    \item $\mathcal{B}_f=\mathcal{G}_f \cup \mathcal{D}_f$;
    \item $\mathcal{C}_f=\mathcal{G}_f \cup (\ell_{\infty} \setminus \mathcal{D}_f)$,
\end{enumerate}
where $\mathcal{D}_f=\{\langle (x,f(x),0) \rangle_{\fqn} \colon x \in \fqn^*\}$ and $\mathcal{G}_f=\{\langle (x,f(x),1) \rangle_{\fqn} \colon x \in \fqn\}$.
The pointset $\mathcal{B}_f$ is an $\fq$-linear blocking set of R\'edei type, whereas the pointset $\C_f$ extends the pointset used to construct translation KM-arcs introduced in \cite[Theorem 2.1]{DeBoeckVdV2016}.
We study their intersections with lines in $\mathrm{PG}(2,q^n)$ and we point out that the pattern of their intersections with lines is strongly related to the weight distribution of the linear set $\mathcal{D}_f$.
This investigation allows us to construct a large number of linear codes for which we can completely describe the weight distribution, also obtaining linear codes with few weights.
Finally we prove that the $\Gamma\mathrm{L}$-class of $\mathcal{D}_f$ can provide an estimation of the number of inequivalent codes one can get from a fixed $\mathcal{B}_f$ or $\mathcal{C}_f$.\\

The paper is organized as follows. Section \ref{sec:prel} contains basic definitions and a brief survey on linear sets and on the relationship between linear codes and multisets of points in a projective space,  in order to make the paper self--contained.  In Section \ref{sec:2constr}, we investigate the pointsets $\mathcal{B}_f$ and $\mathcal{C}_f$ of $\mathrm{PG}(2, q^n)$ by determining their intersections with lines. In Section \ref{sec:connect}, using the well--known  equivalence between linear codes and multisets of a projective space, we describe the parameters of the related codes. Finally, Section \ref{sec:equiv} is devoted to the study of the relation between the projective equivalence of the pointsets $\mathcal{B}_f$ or $\mathcal{C}_f$ and the related codes.

\section{Preliminaries}\label{sec:prel}

We start fixing the following notation. Let $p$ be a prime and $h$ a positive integer. We fix $q=p^h$ and denote by $\fq$ the finite field with $q$ elements. Moreover, we fix a positive integer $n$ and consider the extension field $\fqn$ of degree $n\geq2$ over $\fq$. 
Recall that for the extension $\fqn/\fq$, the \emph{norm} of an element $\alpha \in \fqn$ is defined as
$$ \mathrm{N}_{q^n/q}(\alpha):= \prod_{i=0}^{n-1}\alpha^{q^i},$$
and the \emph{trace} of an element $\alpha \in \fqn$ is defined as
$$ \mathrm{Tr}_{q^n/q}(\alpha):= \sum_{i=0}^{n-1}\alpha^{q^i}.$$

\subsection{Linear sets}

Let $V$ be an $r$-dimensional vector space over $\fqn$ and let $\Lambda=\PG(V,\F_{q^n})=\PG(r-1,q^n)$.
Let $U$ be an $\fq$-subspace of $V$ of dimension $k$, then the set of points
\[ L_U=\{\la {\bf u} \ra_{\mathbb{F}_{q^n}} : {\bf u}\in U\setminus \{{\bf 0} \}\}\subseteq \Lambda \]
is said to be an $\fq$-\emph{linear set of rank $k$}.
Let $P=\langle \mathbf{v}\rangle_{\fqn}$ be a point in $\Lambda$. The \emph{weight of $P$ in $L_U$} is defined as 
\[ w_{L_U}(P)=\dim_{\fq}(U\cap \langle \mathbf{v}\rangle_{\fqn}). \]

We say that an $\fq$-linear set $L_U$ of rank $k$ has \emph{weight distribution} (with respect to $U$) $(i_1,\ldots,i_t)$ with $1\leq i_1 <i_2 < \ldots <i_t\leq k$ if for every $P \in L_U$ 
\[ w_{L_U}(P) \in \{i_1,\ldots,i_t\} \]
and each of  these integers $i_j$ occurs as the weight of at least one point of $L_U$. Also, we will refer to $t$ as the {\it length} of the weight distribution.
Note that if $i_j,i_s \in \{i_1,\ldots,i_t\}$ and $i_j\neq i_s$ then
\begin{equation}\label{eq:pointsum}
i_j+i_s \leq k.
\end{equation}

If $N_{i_j}$ denotes the number of points of $\Lambda$ having weight  $i_j\in \{i_1,\ldots,i_t\}$  in $L_U$  we say that the weight $i_j$ has {\it frequency} $N_{i_j}$. 
The size and  the rank of $L_U$, as well as  the weights and their frequencies, are related by  the following relations:
\begin{equation}\label{eq:card}
    |L_U| \leq \frac{q^k-1}{q-1},
\end{equation}
\begin{equation}\label{eq:pesicard}
    |L_U| =N_{i_1}+\ldots+N_{i_t},
\end{equation}
\begin{equation}\label{eq:pesivett}
   N_{i_1} (q^{i_1-1}+\ldots+q+1)+N_{i_2}(q^{i_2-1}+\ldots+q+1)+\ldots+N_{i_t}(q^{i_t-1}+\ldots+q+1)=q^{k-1}+\ldots+q+1.
\end{equation}
Linear sets attaining the bound in \eqref{eq:card} are called \emph{scattered}.
An $i$-\emph{club} ($1\leq i<n$) in $\PG(1,q^n)$ is an $\F_{q}$-linear set $\mathcal{D}_f$  of rank $n$ in $\PG(1,q^n)$ such that one point has weight $i$ and all the others have weight one,  i.e. $\mathcal{D}_f$ has weight distribution $(1,i)$  and the frequency of the weight $i$ is $1$.
In particular, a $1$-club is a
scattered linear set.
By \eqref{eq:pesicard} and \eqref{eq:pesivett}, we have that an $i$-club has size $q^{n-1}+\ldots+q^i+1$.  Examples of $i$-club of $\PG(1,q^n)$ are known  for $i\in \{n-2,n-1\}$ and $i\in \{r(t-1),r(t-1)+1\}$ if $n=rt$  (see Table \ref{clubpoly}).
As we will recall later, the relevance of clubs arises in the study of a special type of arcs in the projective plane when $q=2$.

Two linear sets $L_U$ and $L_W$ of $\PG (r-1, q^n)$ are said to be $\mathrm{P\Gamma L}$-\emph{equivalent} (or simply \emph{equivalent}) if there is an element $\phi$ in $\mathrm{P\Gamma L}(r,q^n)$ such that $L_U^\phi=L_W$. In the applications it is crucial to have methods to decide whether or not two linear sets are equivalent. For $f\in \mathrm{\Gamma L}(r,q^n)$ we have $L_{U^f}=L_U^{\phi_f}$, where $\phi_f$ denotes the collineation of $\PG (V,\fqn)$ induced by $f$. It follows that if $U$ and $W$ are $\fq$-subspaces of $V$ belonging to the same orbit of $\mathrm{\Gamma L}(r,q^n)$, then $L_U$ and $L_W$ are equivalent. The above condition is only sufficient but not necessary to obtain equivalent linear sets. In the projective line, this fact motivates the following definition from \cite{CsMP}.

Let $L_U$ be an $\mathbb{F}_q-$linear set of $\PG(V,\mathbb{F}_{q^n})=\PG(1,q^n)$ of rank $n$ with maximum field of linearity $\mathbb{F}_q$(\footnote{The \emph{maximum field of linearity} of a linear set is $\mathbb{F}_{q^d}$ if $d$ the greatest positive integer such that $d \mid n$ and $L_U$ is $\mathbb{F}_{q^d}$-linear, i.e. $L_U \neq L_W$ for each $\F_{q^{\ell}}$-subspace $W$ of $\F_{q^n}$ such that $\ell \vert n$ and $\ell>d$. Similar definitions may be given on subspaces and $q$-polynomials.}). We say that $L=L_U$ is of $\mathrm{\Gamma L}$-\emph{class} $s$ if $s$ is the greatest integer such that there exist $\mathbb{F}_q-$subspaces $U_1,\ldots,U_s$ of $V$ with $L_{U_i}=L_U$ for $i \in \{1,\ldots,s\}$ and there is no $f \in \Gamma \mathrm{L}(2,q^n)$ such that $U_i=U_j^f$ for each $i\neq j$, $i,j \in \{1,2,\ldots,s\}$. 
The $\fq$-subspaces $U_1,\ldots,U_s$ are called the $\mathrm{\Gamma L}$-\emph{representatives} of $L$.
If $L_U$ has $\Gamma \mathrm{L}$-class one, then $L_U$ is said to be \emph{simple}.
Note that if $L$ is an $\fq$-linear set of $\Gamma \mathrm{L}$-class $s$ and $U_1, \ldots, U_s$ are $s$ $\Gamma \mathrm{L}$-representatives of $L$, then the $\mathrm{P\Gamma L}(r,q^n)$-orbit of $L$ is obtained by the $s$ $\Gamma\mathrm{L}(r,q^n)$-orbits of $U_1,\ldots,U_s$, namely \[
O_{\mathrm{P\Gamma L}}=\cup_{i=1}^s \{L_W \colon W \in O_{\Gamma\mathrm{L}}(U_i) \}.
\]
So if $L_U$ is a simple linear set, then $L_U$ is projectively equivalent to $L_{U'}$, where $U'$ is an $\F_q$-subspace of rank $n$ if and only if $U$ and $U'$ are $\GammaL$-equivalent.

Moreover, since ${\rm P\Gamma L}(1,q^n)$ is $3$-transitive on $\PG(1,q^n)$, we can assume up to equivalence that $L$ does not contain the point $\langle(0,1)\rangle_{\fqn}$.
This implies that $L$ equals $L_f=L_{U_f}$, where
\[U_{f}=\{(x,f(x)) \colon x \in \fqn\}\] 
for some $q$-polynomial 
\[f(x)\in\mathcal{L}_{n,q}=\left\{ \sum_{i=0}^{n-1} a_ix^{q^i} \colon a_i \in \fqn \right\}.\]
If $L_f$ turns out to be scattered, then we will also say that $f(x)$ is \emph{scattered}, following \cite{John}. 

We refer to \cite{LavVdV} and \cite{Polverino} for comprehensive references on linear sets.

\subsection{Linear codes and (multi)sets of points in finite projective spaces} 

A $q$-\emph{ary linear code} $\C$ is any $\F_q$-subspace of $\F_{q}^m$.
If $\C$ has dimension $r$, we say that $\C$ is an $[m, r]_q$-\emph{code}.
A \emph{generator matrix} $G$ for an $[m,r]_q$-code $\C$ is an $(r \times m)$-matrix over $\F_q$ whose rows form a basis of $\C$, i.e.
\[ \C=\{\mathbf{x}G \colon \mathbf{x}\in \F_q^r\}. \]

We may consider $\F_q^m$ as a metric space endowed with the \emph{Hamming distance} $d$, i.e. $d(\mathbf{x},\mathbf{y})$ is the
number of entries in which $\mathbf{x}$ and $\mathbf{y}$ differ, with $\mathbf{x},\mathbf{y} \in \F_q^m$.
The \emph{Hamming support} of a vector $\mathbf{x} \in \fq^m$ is $\sigma(\mathbf{x})=\{i:x_i \neq 0\}$ and its \emph{Hamming weight} is $w(\mathbf{x})= \lvert \sigma(\mathbf{x}) \rvert$ or equivalently $w(\mathbf{c})=d(\mathbf{c},\mathbf{0})$. 
Recall that the minimum weight of a linear code coincides with its minimum distance. 
A $q$--ary linear code of length $m$, dimension $r$ and minimum distance $d$ is referred to as an $[m, r, d]_q$-code or as an $[m,r]_q$-code if the minimum distance is not relevant/known. A code $\C$ is said to be Hamming non-degenerate if $\cup_{\mathbf{c} \in \C} \sigma(\mathbf{c})=\{1,\ldots,m\}$.
Let denote by $A_i$ the number of codewords of $\C$ with Hamming weight $i$.
The \emph{(Hamming) weight enumerator} is defined as the following polynomial:
\[ 1+A_1 z+\ldots+A_m z^m. \]
This polynomial gives a good deal of information about the code and it is an important invariant under code equivalence, and has been calculated for few families of codes. Also, it is used in the probability theory involved with different ways of decoding.
An $\ell$-\emph{weight code} $\C$ is an $[m,r]_q$-code having $\ell$ nonzero weights $w_1<\ldots<w_{\ell}$, i.e. if the sequence $(A_1,\ldots,A_m)$ have exactly $\ell$ nonzero entries. If $\ell\leq r$ we say that the code has \emph{few weights}.
Much of the focus on linear codes to date has been on codes with few weights, especially on two and three-weight codes, for their applications in secret sharing \cite{DingDing}, authentication codes \cite{aa} and their connections with association schemes \cite{CK1986,CG1984} and with graphs \cite{ShiSo}.

A matrix in $\F_{q}^{m\times m}$ is said to be a \emph{monomial matrix} if it has exactly one nonzero entry in each row and column. Two $[m,r,d]_q$-codes $\C_1$ and $\C_2$ are said to be (monomially) \emph{equivalent} if $\C_2=\C_1 M=\{\mathbf{c}M \colon \mathbf{c} \in \C_1\}$ for some monomial matrix $M \in \F_{q}^{m \times m}$. Note that any monomial matrix $M \in \F_q^{m \times m}$ is the product $M=PD$ where $P \in \F_q^{m \times m}$ is a permutation matrix and $D \in \F_q^{m\times m}$ is a invertible diagonal matrix. Then, the linear codes $\mathcal{C}_1$ and $\mathcal{C}_2$ with generator matrix $G_1$ and $G_2$ respectively, are equivalent if and only if there exist $S \in \mathrm{GL}_r(q)$, $P$ permutation matrix and $D$ non diagonal matrix such that $G_2=SG_1PD$.   

Denote $[m]=\{1,\ldots,m\}$.
Correspondingly, two $[m,r,d]_q$-codes $\C_1$ and $\C_2$ are equivalent if and only if there exists a permutation $\theta:[m] \rightarrow [m]$ and there exist $\beta_1,\ldots,\beta_m \in \F_q^*$ such that
\[
\C_2=\{(\beta_1 x_{\theta(1)},\ldots,\beta_m x_{\theta(m)}) \colon (x_1,\ldots,x_m) \in \C_1\}.
\]

The set of equivalence classes of non-degenerate $[m,r,d]_q$--codes will be denoted by $\classcode$.

A {\em projective $[m, r, d]_q$--system } $(\mathcal{P},\mathrm{m})$ is a finite multiset, where $\mathcal{P} \subseteq \mathrm{PG}(r - 1, q)$ is a set of points  not
all of which lie in a hyperplane, and $\mathrm{m}: \PG(r-1,q) \rightarrow \mathbb{N}$ is the multiplicity function, with $\mathrm{m}(P)>0$ if and only if $P \in \mathcal{P}$ and $\sum_{P \in \mathcal{P}}\mathrm{m}(P)=m$. The parameter $d$ is defined by
\[m - d = \max\left\{\sum_{P \in H} \mathrm{m}(P): H \subseteq \PG(r-1,q), \dim(H)=r-2\right\},\]
see  e.g.\ \cite{land00,TVN}.

Two projective $[m,r,d]_q$--systems $(\mathcal{P},\mathrm{m})$ and $(\mathcal{P}',\mathrm{m}')$  are \emph{equivalent} if there exists a projective isomorphism $\phi \in \mathrm{PGL}(r,q)$ mapping $\mathcal{P}$ to $\mathcal{P}'$ that preserves the multiplicities of the points, i.e.
$\mathrm{m}(P)=\mathrm{m}'(\phi(P))$ for every $P \in \PG(r-1,q)$. The set of all equivalence classes of projective $[m,r,d]_q$ systems will be denoted by $\mathcal{P}[m,r,d]_q$.

There exists a 1-to-1 correspondence between $\classcode$ and $\classsystem$. This correspondence can be formalized by two maps
\[
\Phi:\classcode \longrightarrow \classsystem,
\]

\[
\Psi: \classsystem \longrightarrow \classcode
\]
that are defined as follows.

For a given equivalence class $[\C]$ of non-degenerate $[m,r,d]_q$--codes, consider a generator matrix $G \in \F_{q}^{r \times m}$ of one of the codes in $[\C]$. Let $\mathbf{g}_1,\ldots,\mathbf{g}_m$ be the columns of $G$ and define the set $\mathcal{P}=\{\langle \mathbf{g}_1\rangle_{\fqn},\ldots,\langle\mathbf{g}_m\rangle_{\fqn}\} \subseteq \PG(r-1,q)$. Moreover, define the multiplicity function $\mathrm{m}$ as
\[
\mathrm{m}(P)=\lvert \{i: P=\langle\mathbf{g}_i\rangle_{\fqn}\} \rvert.
\]
Then, the map $\Phi$ is defined as $\Phi([\C])=[(\mathcal{P},\mathrm{m})]$. On the other hand, for a given equivalence class $[(\mathcal{P},\mathrm{m})]$ of projective $[m,r,d]_q$--systems we construct a matrix $G$ by taking as columns a representative of each point $P_i$ in $\mathcal{P}$, counted with multiplicity $\mathrm{m}(P_i)$. Let $\C=\mathrm{rowsp}(G)$, where $\mathrm{rowsp}(G)$ is the space generated by the rows of $G$. We say that $\C$ is the code \emph{associated} with $(\mathcal{P},\mathrm{m})$. We then set $\Psi([\mathcal{P},\mathrm{m}])=[\mathrm{rowsp}(G)]$.
One can easily check the following result.

\begin{theorem} \label{th:correspondenceps}
$\Psi$ and $\Phi$ are well-posed and are the inverse of each other. In particular, there is a one-to-one correspondence between $\classcode$ and $\classsystem$.
\end{theorem}

This correspondence gives also information on the weight distribution of the code associate with a projective system. More precisely, let $\C$ be a non-degenerate $[m,r,d]_q$-code. Let $G \in \F_q^{r \times m}$ be a generator matrix of $\C$. Let $\mathbf{g}_i$, for $i=1,\ldots,m$ be the $i$-th column of $G$. The Hamming weight of a codeword $\mathbf{v}G \in \C$ is
\[
w(\mathbf{v}G)=m-\lvert \{i:\mathbf{v} \cdot \mathbf{g}_i =0\} \rvert 
\]
where $\mathbf{v}^{\perp}=\{\mathbf{z} \colon \mathbf{z} \cdot \mathbf{v}=0 \}$ and $\mathbf{u} \cdot \mathbf{v}=\sum_{i=1}^ku_iv_i$, if $\mathbf{v}=(v_1,\ldots,v_k)$ and $\mathbf{u}=(u_1,\ldots,u_k)$. Let $(\mathcal{P},\mathrm{m})$ be the multiset where $\mathcal{P}=\{[\mathbf{g}_1],\ldots,[\mathbf{g}_m]\} \subseteq \PG(r-1,q)$ and with multiplicity function $\mathrm{m}$ as
\[
\mathrm{m}(P)=\lvert \{i: P=[\mathbf{g}_i]\} \rvert.
\]
We have that the codeword ${\mathbf{v}}G$ has weight $w$ if and only if
the projective hyperplane
\[v_1x_1 + v_2x_2 + \cdots + v_rx_r = 0\]
contains $m  - w$  points of $(\mathcal P, \mathrm{m})$.
So, the number of distinct nonzero weights of $\C$ corresponds to the distinct sizes of the intersections of $({\mathcal P},\mathrm{m})$ with all the hyperplanes.
See also \cite{TVN}.

\section{Two pointsets in $\mathrm{PG}(2,q^n)$ from linear sets in $\mathrm{PG}(1,q^n)$}\label{sec:2constr}

Let $(0\le )m_1 < \cdots < m_s (\le q^{r-1}+q^{r-2} \cdots + q +1)$ be $s$ non--negative integers, a set $\mathcal K$ of points of $\mathrm{PG}(r , q)$ is of  {\em type} $(m_1, \ldots, m_s)$ if  $\vert \pi\cap {\mathcal K}\vert \in \{m_1, \ldots, m_s\}$ for every hyperplane $\pi$ of $\mathrm{PG}(r, q)$ and each of the integers $m_i$ (called {\em intersection numbers}) occurs as the size of the intersection of $\mathcal K$ with a hyperplane of $\mathrm{PG}(r, q)$. 
In this section,  we are going to construct   sets of points of $\mathrm{PG}(2, q^n)$ with few intersection numbers with respect to the lines associated with $\mathbb{F}_q$--linear sets  arising from $q$-polynomials and to describe the parameters of their related codes.

Let $n$ be a positive integer and consider the projective plane $\mathrm{PG}(2, q^n)$ as the projective closure of  $\mathrm{AG}(2, q^n)$ via the line $\ell_\infty= \mathrm{PG}(T, \F_{q^n})$, where $T=\langle (1,0,0),(0,1,0) \rangle_{\F_{q^n}}$.
Let $f(x)$ be any $q$-polynomial in $\mathbb{F}_{q^n}[x]$, then
\[ \mathrm{Im}\left(\frac{f(x)}{x}\right)=\left\{\frac{f(x)-f(y)}{x-y} \colon x,y \in \fqn, x \ne y \right\}, \]
and hence $\mathrm{Im}\left(\frac{f(x)}{x}\right)$ is the set of directions determined by the graph of $f$, that is the set  $$\mathcal{G}_f=\{\langle (x,f(x),1)\rangle_{\fqn} \colon x \in \fqn\}\subseteq \mathrm{AG}(2,q^n).$$
  Also
  $$\mathcal{D}_f=\{\langle (x,f(x),0)\rangle_{\fqn} \colon x \in \fqn^*\}\subseteq \ell_\infty$$ 
is an $\fq$-linear  set of rank $n$ in $\mathrm{PG}(1,q^n)$ and  
$$|\mathcal{D}_f|=\left|\mathrm{Im}\left(\frac{f(x)}{x}\right)\right|.$$  In the case in which $f(x)$ is a $q$-polynomial having $\F_q$ as maximum field of linearity, then the minimum weight of the weight distribution of $\mathcal{D}_f$ is $1$ (see \cite[Proposition 2.3]{CsMP}), and  by the results in \cite{Ball,BBBSSz}, we have the following bounds :
\begin{equation}\label{eq:size}
q^{n-1}+1 \leq \left|\mathrm{Im}\left(\frac{f(x)}x\right)\right|\leq \frac{q^n-1}{q-1}.
\end{equation}
If $\mathcal{D}_f$ is a scattered $\fq$-linear set in $\mathrm{PG}(1,q^n)$, then $|\mathrm{Im}(f(x)/x)|= \frac{q^n-1}{q-1}$, whereas if $f(x)=\mathrm{Tr}_{q^n/q}(x)$ or $f(x)$ is as  in Theorem 5.2 and Corollary 5.3 of \cite{NPSZ2021} then $|\mathrm{Im}(f(x)/x)|= q^{n-1}+1$. So, both the bounds in \eqref{eq:size} are sharp.

Starting from $f(x)$ we can construct two pointsets in $\mathrm{PG}(2,q^n)$ for which we can control the intersection numbers once the weight distribution of $\mathcal{D}_f\subseteq \mathrm{PG}(1,q^n)$ is known:

\begin{enumerate}
    \item $\mathcal{B}_f=\mathcal{G}_f \cup \mathcal{D}_f$;
    \item $\mathcal{C}_f=\mathcal{G}_f \cup (\ell_{\infty} \setminus \mathcal{D}_f)$.
\end{enumerate}

\begin{theorem}\label{th:setfew}
Let $f(x)$ be a $q$-polynomial in $\fqn[x]$ with maximum field of linearity $\F_q$ and suppose that 
\[\mathcal{D}_{f}=\{\langle (x,f(x),0) \rangle_{\fqn} \colon x \in \fqn^*\} \subseteq \ell_{\infty}\] 
has weight distribution $(1,i_2,\ldots,i_t)$. 
Then the following cases occur for $B_f$:
\begin{itemize}
\item [(i)] if $i_t <n-1$, then $\mathcal{B}_f$ is of type $(1,q+1, q^{i_2}+1, \ldots, q^{i_t}+1, \vert \mathcal{D}_f\vert)$, and $\mathcal{B}_f$ admits $t+2$ intersection numbers;
\item [(ii)] if $i_t=n-1$, i.e. the weight distribution of $\mathcal{D}_f$ is $(1,n-1)$, then either $n>2$ and $\mathcal{B}_f$ is of type $(1,q+1,q^{n-1}+1)$, or $n=2$ and $\mathcal{B}_f$ is of type $(1,q+1)$, in both the cases $\mathcal{B}_f$ is $\mathrm{P\Gamma L}(3,q^n)$-equivalent to $\mathcal{B}_{\mathrm{Tr}}$, where $\mathrm{Tr}(x)=\mathrm{Tr}_{q^n/q}(x)$.
\end{itemize}
The following cases occur for $\mathcal{C}_f$:
\begin{itemize}
    \item [(j)] if $q>2$, then $\mathcal{C}_f$ is of type $(0, 2,q, q^{i_2}, \ldots, q^{i_t}, q^{n}+1-\vert \mathcal{D}_f\vert)$, and $\mathcal{C}_f$ admits $t+3$ intersection numbers;
    \item [(jj)] if $q=2$ and $t>2$, then $\mathcal{C}_f$ is of type $(0,q, q^{i_2}, \ldots, q^{i_t}, q^{n}+1-\vert \mathcal{D}_f\vert)$, and $\mathcal{C}_f$ admits $t+2$ intersection numbers;
    \item [(jjj)] if $q=2$, $t=2$ and the weight distribution of $\mathcal{D}_f$ is $(1,i)$, then $\mathcal{C}_f$ is of type $(0,2,2^i)$ if $\mathcal{D}_f$ is an $i$-club in $\mathrm{PG}(1,2^n)$ otherwise $\mathcal{C}_f$ is of type $(0,2,2^i,2^n+1-\vert \mathcal{D}_f\vert)$;
    \item [(jv)] if $q=2$ and $t=1$ then $\mathcal{C}_f$ is of type $(0,q)$ and $\mathcal{D}_f$ is a scattered linear set in $\mathrm{PG}(1,2^n)$;
\end{itemize}
Also, if  $i_t<n-1$ then the number of lines meeting $\mathcal{B}_f$ in $q^{i_j}+1$ points is $q^{n-i_j} N_{i_j}$, while if $q>2$ then the number of lines meeting $\mathcal{C}_f$ in $q^{i_j}$ points is $q^{n-i_j} N_{i_j}$, where $N_{i_j}$ is the frequency of the weight $i_j$.
\end{theorem}
\begin{proof}
Let $\ell$ be any line in $\mathrm{PG}(2,q^n)$.
If $\ell$ meets $\ell_\infty$ in a point of  $\mathcal{D}_f$ of weight $j\in \{1,i_2,\ldots,i_t\}$ then $\ell$ meets $\mathcal{G}_f$ in exactly $q^j$ points or in zero points. If $\ell$ meets $\ell_\infty$ in a point outside $\mathcal{D}_f$ then $\ell$ meets $\mathcal{G}_f$ in exactly one point.   
So, the intersection numbers of $\mathcal{B}_f$ and $\mathcal{C}_f$ belong to 
\[ I_{\mathcal{B}_f}=\{1,q+1, q^{i_2}+1, \ldots, q^{i_t}+1, \vert \mathcal{D}_f\vert\} \]
and
\[ I_{\mathcal{C}_f}=\{0,2,q, q^{i_2}, \ldots, q^{i_t}, q^{n}+1-\vert \mathcal{D}_f\vert\} \]
and all the values occur.
Since $1,i_2,\ldots,i_t$ are distinct and $|\mathcal{D}_f|\geq q^{n-1}$, if $i_t<n-1$, then (i) follows, whereas if $i_t=n-1$, then by \eqref{eq:pointsum} there is one point of weight $n-1$ and all the others have weight one, so (ii) easily follows also by applying \cite[Theorem 5]{LunPol2000}.
Now, if $q^n+1-|\mathcal{D}_f|=q^j$ where $j \in \{1,i_2,\ldots,i_t\}$ then by \eqref{eq:pesicard}
\[ |\mathcal{D}_f|=q^n-q^j+1=N_1+\ldots+N_{i_t}, \]
and by \eqref{eq:pesivett} $\sum_{\ell=1}^t (q^{i_\ell}-1) N_\ell=q^n-1$, so that
\[ (q-2)N_1+\ldots+(q^{i_t}-2)N_{i_t}=q^j-2, \]
and this occurs only if $q=2$ and $t=2$, i.e. $N_j=1$ and hence $\mathcal{D}_f$ is a $j$-club. 
So (j), (jj), (jjj) and (jv) are verified.
\end{proof} 

\begin{remark}
It is possible to determine for each intersection number, the number of lines with that intersection numbers. Corollary \ref{cor:connectcodes} gives these numbers in terms of weights of codewords of the associate linear codes.
\end{remark}

By Theorem \ref{th:setfew} $\fq$-linear sets of the projective line $\PG(1,q^n)$  with {\it short} weight distribution produce  pointsets of $\PG(2,q^n)$ with {\it few} intersection numbers. For instance, starting from  scattered linear sets the previous construction produces well known pointsets in $\PG(2,q^n)$ described in the following corollary (see Table \ref{scattpoly} for the list of known $f(x)$ such that $\mathcal{D}_f$ is scattered, cfr.\ Appendix).
\begin{corollary}\label{cor:scatt}
Let $f(x)$ be a $q$-polynomial in $\fqn[x]$ and suppose that $\mathcal{D}_{f} \subseteq \mathrm{PG}(1,q^n)$ is a scattered $\fq$-linear set in $\mathrm{PG}(1,q^n)$. Then
\begin{itemize}
    \item $\mathcal{B}_f$ is of type $(1,q+1, (q^n-1)/(q-1))$;
    \item $\mathcal{C}_f$ is of type $(0,2,q, q^{n}+1-(q^n-1)/(q-1))$ if $q\ne 2$;
    \item $\mathcal{C}_f$ is of type $(0,2)$ if $q= 2$.
\end{itemize}
\end{corollary} In the case $q=2$ of Corollary \ref{cor:scatt}, $\mathcal{C}_f$ turns out to be a \emph{translation hyperoval}.

In the next  we describe some other families of   $\fq$-linear sets of the projective line admitting exactly two weights.

Applying Theorem \ref{th:setfew} to $i$-clubs we get the following pointsets of $\PG(2,q^n)$ with few intersection numbers.

\begin{corollary}\label{cor:constructionfewweight}
Let $f(x)$ be a $q$-polynomial in $\fqn[x]$ and suppose that $\mathcal{D}_{f} \subseteq \mathrm{PG}(1,q^n)$ is an $i$-club with $i<n-1$. Then
\begin{itemize}
    \item $\mathcal{B}_f$ is of type $(1,q+1,q^{i}+1, q^{n-1}+\ldots+q^i+1)$;
    \item $\mathcal{C}_f$ is of type $(0,2,q, q^{i}, q^{n}-q^{n-1}-\ldots-q^i)$ if $q\neq 2$;
    \item $\mathcal{C}_f$ is of type $(0,2,2^{i})$ if $q=2$ (\cite{DeBoeckVdV2016}).
\end{itemize}
\end{corollary}

When $q=2$, $\mathcal{C}_f$ is of type $(0,2, 2^{i})$ and pointsets of this type are known as \emph{KM-arcs} of type $2^i$, originally introduced in \cite{KM}.
A KM-arc $\mathcal{A}$ in $\PG(2,q^n)$ is called a \emph{translation} KM-arc if there exists a line $\ell$ of $\PG(2,q^n)$ such that the group of elations with axis $\ell$ and fixing $\mathcal{A}$ acts transitively on the points of $\mathcal{A}\setminus \ell$. 
In \cite[Theorem 2.1]{DeBoeckVdV2016} (and Corollary \ref{cor:constructionfewweight}), it was proved that the pointsets $\C_f$ provide examples of translation KM-arcs of type $2^i$ when $D_f$ is an $i$-club in $\mathrm{PG}(2,2^n)$.
The first example of KM-arc of type $2^i$ was presented in \cite{KM} and can be described as $D_{T}$, where $T(x)=\mathrm{Tr}_{q^{\ell m}/q^m}\circ\sigma (x)$; see \cite[Theorem 3.2]{DeBoeckVdV2016}.

Recently, in \cite{NPSZ2021}   $\fq$-linear sets $\mathcal{D}_f$  of the projective line $\PG(1,q^n)$ of rank $n$  with length  of the  weight distribution at most $3$ have been studied. In particular in \cite[Theorem 4.4]{NPSZ2021} it has been proved that if $1<t\leq 3$, $i_1=1$ and  $\mathcal{D}_f$ admits exactly two points of weight bigger than $1$, then  $t=2$, $n$ is even and $i_2=\frac{n}2$. Hence by  \eqref{eq:pesicard} and \eqref{eq:pesivett} the size and the frequences of the weights of $\mathcal{D}_f$ are uniquely determined. Also in the same paper two families of examples of such linear sets have been constructed (see Table \ref{2wpoly} of Appendix).

Applying Theorem \ref{th:setfew} to the examples described in Table \ref{2wpoly} we get the following pointsets of $\PG(2,q^n)$ ($n$ even) with $4$ or $5$ intersection numbers.

\begin{corollary}
Let $f(x)$ be a $q$-polynomial in $\fqn[x]$, $n$ even, and suppose that $\mathcal{D}_{f} \subseteq \mathrm{PG}(1,q^n)$ has two points of weight $n/2$ and all the other points have weight one. Then
\begin{itemize}
    \item $\mathcal{B}_f$ is of type $(1,q+1, q^{n/2}+1, q^{n-1}+\ldots+q^{n/2}-q^{n/2-1}-\ldots-q+1)$;
    \item $\mathcal{C}_f$ is of type $(0,2,q, q^{n/2}, q^{n}-q^{n-1}-\ldots-q^{n/2}+q^{n/2-1}+\ldots+q)$ if $q\ne 2$;
    \item $\mathcal{C}_f$ is of type $(0,2, 2^{n/2}, 2^{n/2+1}-2)$ if $q=2$.
\end{itemize}
\end{corollary}

\section{Connections with linear codes}\label{sec:connect}

Let $f(x)$ be a $q$-polynomial. Let $N_b=|\mathcal{B}_f|$ and $N_c=|\mathcal{C}_f|$.
Denote by $\C_f^b$ the $[N_b,3]_{q^n}$-code associated with the projective $[N_b,3]_{q^n}$-system $(\mathcal{B}_f,\mathrm{m})$, where $\mathrm{m}(P)=1$ if and ony if $P \in \mathcal{B}_f$ and by $\C_f^c$ the $[N_c,3]_{q^n}$-code associated with the projective $[N_c,3]_{q^n}$-system $(\mathcal{C}_f,\mathrm{m})$, where $\mathrm{m}(P)=1$ if and only if $P \in \mathcal{C}_f$.
As consequence of Theorem \ref{th:setfew}, we can determine the weight distribution and the enumerator polynomial of the codes $\C_f^b$ and $\C_f^c$.

\begin{corollary} \label{cor:connectcodes}
Let $f(x)$ be an $q$-polynomial in $\fqn[x]$ and suppose that $\mathcal{D}_{f} \subseteq \mathrm{PG}(1,q^n)$ has weight distribution  $(1,i_2,\ldots,i_t)$ and frequencies of the weights $(N_1,N_{i_2}, \dots, N_{i_t})$.
Then
\begin{enumerate}
    \item if $i_t<n-1$, the weight distribution of $\C_f^b$ is 
    \[(N_{b} - \vert \mathcal{D}_f\vert, N_{b} - q^{i_t}-1, \ldots, N_{b} -q^{i_2}-1, N_{b} -q -1,  N_{b}- 1);\] and in this case the enumerator polynomial of $\C_f^b$ is 
    \[1+ (q^n-1)\left(z^{N_b-|\mathcal{D}_f|}+\sum_{j=1}^t N_{i_j} q^{n-i_j} z^{N_b-q^{i_j}-1}+\left(q^n\cdot (q^n+1-|\mathcal{D}_f|)+\sum_{j=1}^t N_{i_j} (q^n-q^{n-i_j})\right)z^{N_b-1}\right); \]
    \item if $i_t=n-1$ and $n>2$, the weight distribution of $\C_f^b$ is 
    \[(N_{b} - q^{n-1}-1, N_{b} -q-1, N_{b} -1)\] and in this case the enumerator polynomial of $\C_f^b$ is
    \[
    1+ (q^n-1)\left((1+q)z^{N_b-q^{n-1}-1} +q^{n-1}q^{n-1}z^{N_b-q-1} + (q^n(q^n-q^{n-1})+ q^{n-1}(q^n-q^{n-1})+q^n-q )z^{N_b-1} \right)
    \]
    
    \item $i_t=n-1$ and $n=2$, the weight distribution of $\C_f^b$ is 
    \[(N_{b}-q-1, N_{b} -1)\] and in this case the enumerator polynomial of $\C_f^b$ is
    \[
    1+ (q^n-1)\left((q^2+q+1)z^{N_b-q-1}+(q^2(q^2-q)+(q+1)(q^2-q))z^{N_b-1}\right)
    \]
    
    \item If $q>2$, the weight distribution of $\C_f^c$ is 
    \[ (N_{c} - q^{n}-1+\vert \mathcal{D}_f\vert,  N_{c} - q^{i_t}, \ldots,  N_{c} - q^{i_2}, N_{c} -q, N_{c} - 2, N_{c}); \]
  the enumerator polynomial of $\C_f^c$ is 
    \[ 1+ (q^n-1)\left(z^{N_c-q^n-1+|\mathcal{D}_f|}+\sum_{j=1}^t N_{i_j} \cdot q^{n-i_j} z^{N_c-q^{i_j}}
    + q^n\cdot (q^n+1-|\mathcal{D}_f|) z^{N_c-2}
    +\right.\] \[\left. \left(\sum_{j=1}^t N_{i_j} (q^n-q^{n-i_j})\right)z^{N_c}\right). \]
    \item If $q=2$ and $t>2$, the weight distribution of $\C_f^c$ is 
    \[ (N_{c} - q^{n}-1+\vert \mathcal{D}_f\vert,  N_{c} - q^{i_t}, \ldots,  N_{c} - q^{i_2}, N_{c} -q, N_{c}); \]
  the enumerator polynomial of $\C_f^c$ is 
    \[ 1+  (q^n-1)\left(z^{N_c-q^n-1+|\mathcal{D}_f|}+\sum_{j=2}^t N_{i_j} \cdot q^{n-i_j} z^{N_c-q^{i_j}} 
    + (q^n\cdot (q^n+1-|\mathcal{D}_f|)+ N_1 \cdot q^{n-1}) z^{N_c-q}\right)\]
\[+ (q^n-1)\left(\left(\sum_{j=1}^t N_{i_j} (q^n-q^{n-i_j})\right)z^{N_c} \right). \]
    \item If $q=2$ and the weight distribution of $\mathcal{D}_f$ is $(1,i)$ and $\mathcal{D}_f$ is an $i$-club, the weight distribution of $\C_f^c$ is 
    \[ (N_{c} - q^{i}, N_{c} -q, N_{c}); \]
  the enumerator polynomial of $\C_f^c$ is 
    \[ 1+ (q^n-1)\left((N_iq^{n-i}+1)z^{N_c-q^i}+(q^n(q^n+1-\lvert \mathcal{D}_f\rvert) +N_1q^{n-1})z^{N_c-q}+\right. \]\[\left.(N_i(q^n-q^{n-i})+N_1(q^n-q^{n-1}))z^{N_c}\right). \]
    If $q=2$ and the weight distribution of $\mathcal{D}_f$ is $(1,i)$ and $\mathcal{D}_f$ is not an $i$-club, the weight distribution of $\C_f^c$ is
    \[ (N_{c} - q^{n}-1+\vert \mathcal{D}_f\vert,  N_{c} - q^{i}, N_{c} -q, N_{c}); \]
  the enumerator polynomial of $\C_f^c$ is 
    \[ 1+  (q^n-1)\left(z^{N_c-q^n-1+|\mathcal{D}_f|}+ N_{i} \cdot q^{n-i} z^{N_c-q^{i}} 
    + (q^n\cdot (q^n+1-|\mathcal{D}_f|)+ N_1 \cdot q^{n-1}) z^{N_c-q}\right)\]
\[+\left(\sum_{j=1}^t N_{i_j} (q^n-q^{n-i_j})\right)z^{N_c}. \]
    \item If $q=2$ and $\mathcal{D}_f$ is a scattered linear set, the weight distribution of $\C_f^c$ is 
    \[ ( N_{c} -q, N_{c}); \]
  the enumerator polynomial of $\C_f^c$ is 
    \[ 1+ (q^n-1)\left(1+(q^{n+1} +(q^n-1)q^{n-1})z^{N_c-q}+(q^n-1)(q^n-q^{n-1})z^{N_c}\right). \]
\end{enumerate}
\end{corollary}
\begin{proof}
Let us start by analyzing the case (1). As said in Section 2 we have that the codeword ${\mathbf{v}}G$ has weight $N_b-w$ if and only if
the projective line
\[v_1x_1 + v_2x_2 + v_3x_3 = 0\]
contains $w$  points of $\mathcal{B}_f$. So the number of codewords of $\mathcal{C}_f$ of weight $N_b-w$ coincides with the number of lines that intersect $\mathcal{B}_f$ in $w$ points.
Only the R\'edei line (the line at infinity) intersects $\mathcal{B}_f$ in $\lvert \mathcal{D}_f \rvert$ points. Let $\ell$ be any line in $\mathrm{PG}(2,q^n)$.
If $\ell$ meets $\ell_\infty$ in a point of  $\mathcal{D}_f$ of weight $j\in \{1,i_2,\ldots,i_t\}$ then $\ell$ meets $\mathcal{G}_f$ in exactly $q^j$ points or in zero points. Then there are $q^{n-j}$ lines through a point of weight $j\in \{1,i_2,\ldots,i_t\}$ that intersect $\mathcal{B}_f$ in $q^j+1$ points and hence $q^n-q^{n-j}$ lines through a point of weight $j\in \{1,i_2,\ldots,i_t\}$ that intersect $\mathcal{B}_f$ in exactly one point. The remaining ones intersect the $\mathcal{B}_f$ in one point. \\
By similar arguments, the other cases follow.
\end{proof}

\begin{remark}
In the above result we actually completely determine the pattern of the intersections between lines and the two pointsets. 
\end{remark}

\section{Examples from the known families of linearized polynomials}

In this section, we exhibit the enumerator polynomials of $\mathcal{C}_f^b$ and $\mathcal{C}_f^c$ when $f(x)$ defines an $i$-club, a scattered linear set or a linear set having two points of weight $n/2$ and all the other points of weight one in $\PG(1,q^n)$.
We will use the following notation $\theta_i=\frac{q^{i+1}-1}{q-1}$ and recall that $N_b=|\mathcal{B}_f|$ and $N_c=|\mathcal{C}_f|$.

\subsection*{$i$-clubs}
If $\mathcal{D}_f$ is an $i$-club, then $N_b=q^n+q^{n-1}+\ldots+q^i+1$ and $N_c=2q^n-q^{n-1}-\ldots-q$.
\begin{enumerate}
    \item If $i<n-1$, the weight distribution of $\C_f^b$ is 
    \[(N_{b} -(\theta_{n-1}-\theta_{i-1}) -1, N_{b} - q^{i}-1, N_{b} -q -1,  N_{b}- 1),\] and in this case the enumerator polynomial of $\C_f^b$ is 
     \[ 1+(q^{n}-1)(z^{N_b-(\theta_{n-1}-\theta_{i-1})-1}+q^{n-i} z^{N_b-q^{i}-1}+(\theta_{n-1}-\theta_{i-1})q^{n-1} z^{N_b-q-1} )\]
    \[+(q^{n}-1)(\left(q^n\cdot (q^n-\theta_{n-1} +\theta_{i-1})+ (q^n-q^{n-i})+(\theta_{n-1}-\theta_{i-1})(q^n-q^{n-1})\right)z^{N_b-1}), \]
    
    \item if $i=n-1$ and $n>2$, the weight distribution of $\C_f^b$ is 
    \[(N_{b} - q^{n-1}-1, N_{b} -q-1, N_{b} -1)\] and in this case the enumerator polynomial of $\C_f^b$ is
    \[
    1+(q^{n}-1)((1+q)z^{N_b-q^{n-1}-1} +q^{n-1}q^{n-1}z^{N_b-q-1} + (q^n(q^n-q^{n-1})+ q^{n-1}(q^n-q^{n-1})+q^n-q )z^{N_b-1}).
    \]
    \item If $q>2$, the weight distribution of $\C_f^c$ is 
    \[ (N_{c} - q^{n}-1+\vert \mathcal{D}_f\vert,  N_{c} - q^{i}, N_{c} -q, N_{c} - 2, N_{c}), \]
  the enumerator polynomial of $\C_f^c$ is 
    \[ 1+(q^n-1)(z^{N_c-q^n+ \theta_{n-1} -\theta_{i-1}}+ q^{n-i} z^{N_c-q^{i}} + (\theta_{n-1} -\theta_{i-1})q^{n-1}z^{N_c-q} \]\[
    + q^n\cdot (q^n-\theta_{n-1} +\theta_{i-1}) z^{N_c-2}
    +\left( q^n-q^{n-i}+(\theta_{n-1}-\theta_{i-1}+1)(q^n-q^{n-1})\right)z^{N_c}). \]

    \item If $q=2$, as in (7) of Corollary \ref{cor:connectcodes}.
    
    
\end{enumerate}
    
\subsection*{Scattered linear sets}
    
If $\mathcal{D}_f$ is a scattered linear set, then $N_b=\frac{q^{n+1}-1}{q-1}$ and $N_c=2q^n-q^{n-1}-\ldots-q$.
\begin{enumerate}
    \item If $n>2$, the weight distribution of $\C_f^b$ is 
    \[(N_{b} - \vert \mathcal{D}_f\vert, N_{b} -q -1,  N_{b}- 1),\] 
    and in this case the enumerator polynomial of $\C_f^b$ is 
    \[ 1 + (q^n-1)(z^{N_b-\theta_{n-1}}+\theta_{n-1}\cdot q^{n-1} z^{N_b-q-1}+\left(q^n\cdot \left(q^n+1-\theta_{n-1}\right)+\theta_{n-1} (q^n-q^{n-1})\right)z^{N_b-1}), \]
    
    \item  $n=2$, the weight distribution of $\C_f^b$ is 
    \[(N_{b}-q-1, N_{b} -1)\] and in this case the enumerator polynomial of $\C_f^b$ is
    \[
    1+(q^n-1)(\theta_2 \cdot z^{N_b-q-1}+(q^2(q^2-q)+(q+1)(q^2-q))z^{N_b-1})
    \]
    
    \item If $q>2$, the weight distribution of $\C_f^c$ is 
    \[ (N_{c} - q^{n}-1+\theta_{n-1},  N_{c} - q, N_{c} - 2, N_{c}); \]
  the enumerator polynomial of $\C_f^c$ is 
    \[ 1+(q^n-1)(z^{N_c-q^n-1+\theta_{n-1}}+\theta_{n-1} \cdot q^{n-1} z^{N_c-q} 
    + q^n\cdot (q^n+1-\theta_{n-1}) z^{N_c-2}
    +\theta_{n-1} \cdot (q^n-q^{n-1}) z^{N_c}). \]
    \item If $q=2$, as in (7) of Corollary \ref{cor:connectcodes}.

\end{enumerate}

\subsection*{Linear sets with two points of weight $n/2$ and all the others of weight one}    
If $n$ is even, $\mathcal{D}_f$ has two points of weight $n/2$ with $n>2$ and all the other points have weight one, then $N_b=q^n+q^{n-1}+\ldots+q^{n/2}-q^{n/2-1}-\ldots-q+1$ and $N_c=2q^n-q^{n-1}-\ldots-q^{n/2}+q^{n/2-1}+\ldots+q$.
    \begin{enumerate}
    \item The weight distribution of $\C_f^b$ is 
    \[(N_{b}-(\theta_{n-1}-2(\theta_{n/2-1}-1)), N_{b} - q^{n/2}-1, N_{b} -q -1,  N_{b}- 1);\] and in this case the enumerator polynomial of $\C_f^b$ is 
     {\small
    \[ 1+(q^n-1)(z^{N_{b}-(\theta_{n-1}-2(\theta_{n/2-1}-1))}+ 2 q^{n/2} z^{N_b-q^{n/2}-1} + (\theta_{n-1}-2\theta_{n/2-1})q^{n-1} z^{N_b-q-1}\]
    \[+\left(q^n\cdot (q^n+1-(\theta_{n-1}-2(\theta_{n/2-1}-1))+ 2 (q^n-q^{n/2})+\right) \]\[\left( (\theta_{n-1}-2\theta_{n/2-1})(q^n-q^{n-1}) \right)z^{N_b-1}). \]}
    
    \item If $q>2$, the weight distribution of $\C_f^c$ is 
    \[ (N_{c} - q^{n}-1+\theta_{n-1}-2(\theta_{n/2-1}-1),  N_{c} - q^{n/2},N_{c} -q, N_{c} - 2, N_{c}); \]
  the enumerator polynomial of $\C_f^c$ is 
    \[ 1+(q^n-1)(z^{(N_{c} - q^{n}-1+\theta_{n-1}-2(\theta_{n/2-1}-1)}+\]
    \[2 q^{n/2}z^{N_c-q^{n/2}}+ (\theta_{n-1}-2\theta_{n/2-1})(q^n-q^{n-1}) q^{n-1} z^{N_c-q} \]
    \[+ q^n\cdot (q^n+1-\theta_{n-1}+2) z^{N_c-2}\]
    \[+\left(2 (q^n-q^{n/2})+(\theta_{n-1}-2\theta_{n/2-1})(q^n-q^{n-1})\right)z^{N_c}). \]
    
    \item If $q=2$, as in (6) of Corollary \ref{sec:connect}.
    
\end{enumerate}

\section{Equivalence}\label{sec:equiv}

In this section we determine the relationship between $q$-polynomials defining equivalent linear codes.
To this aim recall that $q=p^h$ and also the following definition from \cite{CsMP2019}. Two $q$-polynomials $f(x)$ and $g(x)$ in $\mathcal{L}_{n,q}$ are \emph{equivalent} if the $\fq$-subspaces 
 $$U_f=\{(x,f(x))\,:\, x\in \fqn\} \,\,\mbox{and} \,\,  U_g=\{(x,g(x))\,:\, x\in \fqn\}$$ 
are equivalent under the action of the group $\Gamma {\mathrm L}(2,q^n)$, i.e.\ if there exists an element $\varphi\in\Gamma {\mathrm L}(2,q^n)$ such that $U_f^\varphi=U_g$.
Moreover, if $U_f$ and $U_g$ are equivalent under the action of $\mathrm{GL}(2,q^n)$ we say that $f(x)$ and $g(x)$ are \emph{linearly equivalent}.

Note that if $f(x)$ is a $q$-polynomial in $\mathcal{L}_{n,q}$, then the pointset  
$$\mathcal{B}_f=\mathcal{G}_f \cup \mathcal{D}_f=\{\langle(x,f(x),\alpha)\rangle_{\fqn} \, : \, x\in \fqn,\, \alpha\in \fq, (x,\alpha)\ne(0,0)\},$$ 
defined in Section 3, is an $\fq$-linear set of rank $n+1$ in  $\PG(2,q^n)$. Such a set of points is an {\it  $\fq$-linear blocking set} of the projective plane $\PG(2,q^n)$  and 
since  $|\mathcal{B}_f\setminus \ell_\infty|=|\mathcal{B}_f\setminus \mathcal{D}_f|=q^n$, $\mathcal{B}_f$ is a blocking set  {\it of R\'edei type}  having  the line $\ell_\infty$ as a  {\it R\'edei  line}.
Also every line of  $\PG(2,q^n)$ intersects $\mathcal{B}_f$ in a number of points congruent to 1 modulo $q$ and hence $\mathcal{B}_f$ has exponent at least $h$, where the {\it exponent} of $\mathcal{B}_f$ is
the maximal integer $e$  ($0 \leq e \leq hn)$ such that $|\ell \cap \mathcal{B}_f|\equiv \, 1 \, (mod \,\, p^e)$  for every line $ \ell$ in
$\PG(2, q^n)$ (see e.g. \cite[Section 3.1]{Polverino} for more details on linear blocking sets). Starting from an $\fq$-linear blocking set of R\'edei type $\mathcal{B}$ with R\'edei line $\ell$ we can consider the following set
\[ \mathcal{C}=(\mathcal{B}\setminus \ell)\cup (\ell \setminus \mathcal{B}). \]
The pointsets in $\mathrm{PG}(2,q^n)$ obtained in this way will be called \emph{co-blocking sets of R\'edei type} with R\'edei line $\ell$.
The construction $\mathcal{C}_f$ is such an example.
Following the arguments used in the proof of \cite[Theorem 5.5]{CsMP} we may state the following result.

\begin{theorem}\label{th:clasB_f}
Let $f(x)$ and $g(x)$ be $q$-polynomials in $\mathcal{L}_{n,q}$ admitting $\fq$ as maximum field of linearity.
Then $\mathcal{B}_f$ and $\mathcal{B}_g$ are $\mathrm{P\Gamma L}(3,q^n)$-equivalent if and only if the polynomials $f(x)$ and $g(x)$ are equivalent.
Also, $\mathcal{B}_f$ and $\mathcal{B}_g$ are $\mathrm{PGL}(3,q^n)$-equivalent if and only if $f(x)$ and $g(x)$ are linearly equivalent.
\end{theorem}
\begin{proof}
First of all note that if $\fq$ is the maximum field of linearity of the polynomial $f(x)$, then $\fq$ is the maximum field of linearity of the $\fq$-subspace $U_f$ and hence by \cite[Proposition 2.3]{CsMP} there exists a point $P$ of $\ell_\infty$ of weight one in $\mathcal{B}_f$. This means that there exists at least one line of $\PG(2,q^n)$ through the point $P$ intersecting $\mathcal{B}_f$ in exactly $q+1$ points, i.e. the blocking set $\mathcal{B}_f$ has exponent $h$, and by \cite[Theorem 4.1]{DeBeule}.
If $\mathcal{B}_f$ has more than one R\'edei line, then also $\mathcal{B}_g$ admits more than one R\'edei line as well.  By \cite[Theorem 5]{LunPol2000} it follows that $\mathcal{B}_f$ and $\mathcal{B}_g$ are $\mathrm{PGL}(3,q^n)$-equivalent to $\mathcal{B}_{\mathrm{Tr}_{q^n/q}(x)}$. Recall that the size of $\mathcal{B}_{\mathrm{Tr}_{q^n/q}(x)}$ is $q^n+q^{n-1}+1$ and all the R\'edei lines of $\mathcal{B}_{\mathrm{Tr}_{q^n/q}(x)}$ pass through the point $\langle (1,0,0) \rangle_{\fqn}$ which has weight $n-1$ in $\mathcal{B}_{\mathrm{Tr}_{q^n/q}(x)}$.
This implies that the R\'edei lines of $\mathcal{B}_f$ meet $\mathcal{B}_f$ in an $\fq$-linear set of rank $n$ of size $q^{n-1}+1$ and containing a point of weight $n-1$, that is an $(n-1)$-club.
Since $\ell_\infty$ is a R\'edei line, then by \cite[Theorem 2.3]{DeBoeckVdV2016}, we have that $U_f$ is $\mathrm{GL}(2,q^n)$-equivalent to $U_{\mathrm{Tr}_{q^n/q}(x)}$ and the assertion is proved.
Now, suppose that both $\mathcal{B}_f=L_{V_1}$ and $\mathcal{B}_g=L_{V_2}$ have only one R\'edei line, which is the line $\ell_{\infty}$.
In particular, 
\[ V_1=\overline{U}_f \oplus \langle (0,0,1)\rangle_{\fq}\,\,\text{and}\,\, V_2=\overline{U}_g \oplus \langle (0,0,1)\rangle_{\fq}, \]
where
\[ \overline{U}_f=\{(x,f(x),0)\colon x \in \fqn\} \]
and 
\[ \overline{U}_g=\{(x,g(x),0)\colon x \in \fqn\}. \]
So, $\mathcal{B}_f\cap \ell_{\infty}=L_{\overline{U}_f}$ and $\mathcal{B}_g\cap \ell_{\infty}=L_{\overline{U}_g}$.
Suppose that $\mathcal{B}_f$ and $\mathcal{B}_g$ are $\mathrm{P\Gamma L}(3,q^n)$-equivalent. 
By \cite[Proposition 2.3]{BoPol}, there exists $\varphi \in\mathrm{\Gamma L}(3,q^n)$ such that $\varphi(V_1)=V_2$. 
Since $\ell$ is the unique R\'edei line, we also have $\varphi(\overline{U}_f)=\overline{U}_g$, i.e. $U_f$ and $U_g$ are $\mathrm{\Gamma L}(2,q^n)$-equivalent.
Conversely, suppose that $U_f$ and $U_g$ are $\mathrm{\Gamma L}(2,q^n)$-equivalent via a map $\mu \in \mathrm{\Gamma L}(2,q^n)$. Then we may extend $\mu$ to an element $\varphi \in \mathrm{\Gamma L}(3,q^n)$ such that  $\varphi(\overline{U}_f)=\overline{U}_g$ and $\varphi((0,0,1))=(0,0,1)$, so that $\varphi(V_1)=V_2$ and therefore $\mathcal{B}_f$ and $\mathcal{B}_g$ are $\mathrm{P\Gamma L}(3,q^n)$-equivalent.
The last part can be proved in a similar way.
\end{proof}

\begin{remark}
In \cite[Theorem 5]{LunPol2000}, the following is stated. Let $\mathcal{B}$ be an $\fq$-linear  blocking  set of  $\PG(2,q^n)$ of  size $|\mathcal{B}|\geq q^n+q^{n-1}+1$. If $\mathcal{B}$  has  at  least  two  R\'edei  lines  then  it is  $\mathrm{P\Gamma L}(3,q^n)$-equivalent to the blocking set $\mathcal{B}_{\mathrm{Tr}_{q^n/q}(x)}$.
However, following its proof, one can replace $\mathrm{P\Gamma L}(3,q^n)$ by $\mathrm{PGL}(3,q^n)$.
\end{remark}

As a consequence of the previous result  we get the following theorem.

\begin{theorem}\label{th:GLclass}
Let $f(x)$ be a $q$-polynomial in $\mathcal{L}_{n,q}$ having $\fq$ as maximum field of linearity and let $s$ be the $\mathrm{\Gamma L}$-class of $\mathcal{D}_f$.
Then $s$ is the number of $\mathrm{P\Gamma L}(3,q^n)$-inequivalent $\fq$-linear blocking sets of the form  $\mathcal B_g$ of $\PG(2,q^n)$ containing $\mathcal{D}_f$.
\end{theorem}
\begin{proof}
Let  $U_1,U_2,\dots, U_s$ be $s$ distinct $\GammaL$-representatives of $\mathcal{D}_f$. Since $\langle (0,1)\rangle_{\fqn}\not \in \mathcal{D}_f$,  we may write 
$$U_i=U_{g_i}=\{(x, g_i(x), 0)\,:\, x\in \fqn\}$$ 
where the $g_i(x)$'s are pairwise inequivalent $q$-polynomials. 
Hence, since $\mathcal{D}_f=L_{U_{i}}=\mathcal{D}_{g_i}$ we get
$$\mathrm{Im}\left(\frac{f(x)}x\right)=\mathrm{Im}\left(\frac{g_i(x)}x\right)$$ for every $i\in \{1,\ldots,s\}$. So, by \cite[Proposition 2.1]{CsMP2019}, $\fq$  is the maximum field of linearity of $g_i(x)$  for any $i\in \{1,\dots,s\}$. 
Now, we have
$$\mathcal{B}_{g_i}=\{\langle (x,g_i(x),\alpha)\rangle_{\fqn} \, : \, x\in \fqn,\, \alpha\in \fq, (x,\alpha)\ne(0,0)\},$$ 
for every $i\in \{1,\ldots,s\}$. 
Noting that $\mathcal{B}_{g_i} \cap \ell_{\infty}=\mathcal{D}_{g_i}=\mathcal{D}_f$,  Theorem \ref{th:clasB_f} implies the assertion.
\end{proof}

Therefore, by Theorems \ref{th:clasB_f}, \ref{th:GLclass} and \ref{th:correspondenceps}  we get the following corollary.

\begin{corollary} \label{cor:CbEquiv}
Let $f(x)$ and $g(x)$ be two $q$-polynomials in $\mathcal{L}_{n,q}$ admitting $\fq$ as maximum field of linearity.
Then $\mathcal{C}^b_f$ and $\mathcal{C}^b_g$ are  (monomially) equivalent if and only if the polynomials $f(x)$ and $g(x)$ are linearly equivalent.
Also, if the $\GammaL$-class of  $\mathcal{D}_f$ is  $s$ and $D_{g_i}$, for $i\in \{1,\dots,s\}$,  are $s$ of its representatives, then the codes  $\mathcal{C}^b_{g_i}$, $i\in \{1,\ldots,s\}$ are pairwise inequivalent. 
\end{corollary}

\begin{remark}
Consider $f(x)=x^{q}\in \mathcal{L}_{n,q}$. The $\mathrm{\Gamma L}$-class of $\mathcal{D}_f$ is $\varphi(n)/2$ (see \cite[Remark 5.6]{CsMP}), where $\varphi$ is the Euler totient function, and its representatives are 
\[U_i=\{(x,x^{q^i}) \colon x \in \fqn\}\]
where $i \in \{1,\ldots, \lfloor n/2 \rfloor\}$ and $\gcd(i,n)=1$.
Therefore, by the above corollary we obtain at least $\varphi(n)/2$ inequivalent codes arising from $\mathcal{D}_{x^q}$.
\end{remark}

By Theorem \ref{th:clasB_f} and  \ref{th:setfew} we obtain the following result.

\begin{proposition}\label{prop:clasC_f}
Let $q>2$ and let $f(x)$ and $g(x)$ be two $q$-polynomials in $\mathcal{L}_{n,q}$ admitting $\fq$ as the maximum field of linearity.
Then $\mathcal{C}_f$ and $\mathcal{C}_g$ are $\mathrm{P\Gamma L}(3,q^n)$-equivalent if and only if the polynomials $f(x)$ and $g(x)$ are equivalent.
Also, $\mathcal{C}_f$ and $\mathcal{C}_g$ are $\mathrm{PG L}(3,q^n)$-equivalent if and only if the polynomials $f(x)$ and $g(x)$ are linearly equivalent.
\end{proposition}
\begin{proof}
Since $q>2$ and $\fq$ is the maximum field of linearity, by Theorem \ref{th:setfew}  the line $\ell_{\infty}$ is the line of $\PG(2,q^n)$ containing the maximum number of points of $\mathcal{C}_f$ and $\mathcal{C}_g$. Hence, if an element $\varphi \in \mathrm{P \Gamma L}(3,q^n)$ (or $\varphi \in \mathrm{PGL}(3,q^n)$) maps $\mathcal{C}_f$ in  $\mathcal{C}_g$, then $\varphi$ fixes $\ell_{\infty}$ and hence  $\varphi(\mathcal{D}_f)= \mathcal{D}_g$, i.e.    $\varphi(\mathcal{B}_f)= \mathcal{B}_g$, so the assertion follows by Theorem \ref{th:clasB_f}.
\end{proof}

By the previous proposition and by Theorem  \ref{th:GLclass} similarly to the case of blocking sets, we obtain the following results.

\begin{proposition}
Let $f(x)$ be a $q$-polynomial in $\mathcal{L}_{n,q}$ with $q>2$ and having $\fq$ as maximum field of linearity. Let $s$ be the $\mathrm{\Gamma L}$-class of $\mathcal{D}_f$.
Then $s$ is the number of $\mathrm{P\Gamma L}(3,q^n)$-inequivalent $\fq$-linear co-blocking sets of the form $\mathcal C_g$ of $\PG(2,q^n)$ containing $\ell_{\infty} \setminus \mathcal{D}_f$.
\end{proposition}

\begin{corollary} \label{cor:CcEquiv}
Let $q>2$ and let $f(x)$ and $g(x)$ be $q$-polynomials in $\mathcal{L}_{n,q}$ admitting $\fq$ as maximum field of linearity.
Then $\mathcal{C}^c_f$ and $\mathcal{C}^c_g$ are  (monomially) equivalent if and only if the polynomials $f(x)$ and $g(x)$ are linearly equivalent.
Also, if  $s$ the $\GammaL$-class of  $\mathcal{D}_f$ and $U_{g_i}$, $i\in \{1,\dots,s\}$, are $s$ of its representatives, then the codes $\mathcal{C}^c_{g_i}$, with $i\in \{1,\dots,s\}$, are pairwise inequivalent.  
\end{corollary}

Finally, by Corollaries \ref{cor:CbEquiv} and \ref{cor:CcEquiv} we have that the known inequivalent $q$-polynomials defining linear sets  on the projective line with few weights give rise to inequivalent  linear codes with few weights.

\begin{corollary}
The codes $\mathcal{C}_f^b$ and $\C_{f}^c$ constructed from the examples of $\fq$-polynomials given in Tables \ref{scattpoly}, \ref{clubpoly} and \ref{2wpoly} are pairwise inequivalent (cfr. Appendix).
\end{corollary}

Note that the weight distributions of the codes arising from Tables \ref{scattpoly}, \ref{clubpoly} and \ref{2wpoly} are described in Section \ref{sec:connect} (see also the appendix).

\section*{Acknowledgements}
We want to thank the referees for their careful reading and for their valuable comments. 
The research was supported by the project ``VALERE: Vanvitelli pEr la RicErca" of the University of Campania ``Luigi Vanvitelli'' and was partially supported by the Italian National Group for Algebraic and Geometric Structures and their Applications (GNSAGA - INdAM).

\section*{Appendix}

In the following tables we list the known examples of $q$-polynomials $f(x)$ such that $\mathcal{D}_f$ is scattered, an $i$-club or it has only two points of weight greater than one.
Different entries of the tables correspond to $\Gamma \mathrm{L}(2,q^n)$-inequivalent subspaces $U_f=\{(x,f(x)) \,:\, x\in \fqn\}$.

\begin{table}[htp]
\tabcolsep=0.2 mm
\begin{tabular}{|c|c|c|c|c|c|}
\hline
\hspace{0.5cm} & \hspace{0.2cm}$n$\hspace{0.2cm} & $f(x)$ & \mbox{conditions} & \mbox{references} \\ \hline
i) & & $x^{q^s}$ & $\gcd(s,n)=1$ & \cite{BL2000} \\ \hline
ii) & & $x^{q^s}+\delta x^{q^{s(n-1)}}$ & $\begin{array}{cc} \gcd(s,n)=1,\\ \mathrm{N}_{q^n/q}(\delta)\neq 1 \end{array}$ & \cite{LunPol2000,LMPT2015}\\ \hline
iii) & $2\ell$ & $x^{q^s}+x^{q^{s(\ell-1)}}+\delta^{q^\ell+1}x^{q^{s(\ell+1)}}+\delta^{1-q^{2\ell-1}}x^{q^{s(2\ell-1)}}$ & $\begin{array}{cc} q \hspace{0.1cm} \text{odd}, \\ \mathrm{N}_{q^{2\ell}/q^\ell}(\delta)=-1,\\ \gcd(s,t)=1 \end{array}$ & \cite{BZZ,LMTZ,LZ2,NSZ,ZZ}\\ \hline
iv) & $6$ & $x^q+\delta x^{q^{4}}$  &  $\begin{array}{cc} q>4, \\ \text{certain choices of} \, \delta \end{array}$ & \cite{CMPZ,BCsM,PZ2019} \\ \hline
v) & $6$ & $x^{q}+x^{q^3}+\delta x^{q^5}$ & $\begin{array}{cccc}q \hspace{0.1cm} \text{odd}, \\ \delta^2+\delta =1 \end{array}$
 & \cite{CsMZ2018,MMZ} \\ \hline
vi) & $8$ & $x^{q}+\delta x^{q^5}$ & $\begin{array}{cc} q\,\text{odd},\\ \delta^2=-1\end{array}$ & \cite{CMPZ} \\ \hline
\end{tabular}
\caption{Known examples of scattered  $\mathcal{D}_{f}$}
\label{scattpoly}
\end{table}

\begin{table}[htp]
\tabcolsep=0.2 mm
\begin{tabular}{|c|c|c|c|c|c|}
\hline
\hspace{0.5cm} & \hspace{0.2cm}$n$\hspace{0.2cm} & \hspace{0.2cm}$i$\hspace{0.2cm}& $f(x)$ & \mbox{conditions} & \mbox{references} \\ \hline
i) & $rt$ & $r(t-1)$ & $\mathrm{Tr}_{q^{rt}/q^r}\circ x^{q^s}$ & $\gcd(s,n)=1$ & \cite{DeBoeckVdV2016,KM}\\ \hline
ii) &  & $n-2$ & $\mathrm{Tr}_{q^n/q}(b_{n-2}x)+\lambda\mathrm{Tr}_{q^n/q}(b_{n-1}x)$ & 
$\begin{array}{ccc}
\gcd(s,n)=1\\
\fq(\lambda)=\fqn\\
\mathcal{B}=(1,\lambda,\ldots,\lambda^{n-3},\lambda^{n-2}+\omega,\omega\lambda)\\
\mathcal{B}\,\, \fq\text{-basis of }\fqn\\
(b_0,\ldots,b_{n-1})\,\,\,\text{dual basis of } \mathcal{B}
\end{array}$ & \cite{DeBoeckVdV2016,NPSZ202x}\\ \hline
iii) & $rt$ & $r(t-1)$ & $f\left( \mathrm{Tr}_{q^n/q^r}(c_0x) \right)-a\mathrm{Tr}_{q^n/q^r}(c_0x)$ & 
$\begin{array}{ccc} 
f(x) \in \mathcal{L}_{r,q}\,\, \text{scattered}\\
\{1,\omega,\ldots,\omega^{t-1}\}\,\,\, \F_{q^r}\text{-basis of }\F_{q^n}\\
c_0=\frac{1}{g'(\omega)}\sum_{j=0}^{t-1}\omega^j a_{i+j+1}\\
g(x)=a_0+a_1x+\ldots+a_{t-1}x^{t-1}+x^t\\ 
g(x)\text{ minimal polynomial of }\omega\text{ over } \F_{q^r}\\
f(x)-ax\,\, \text{is invertible over}\,\,\F_{q^r}\\
\end{array}$
& \cite{DeBoeckVdV2016,GW2003,NPSZ202x}\\ \hline
iv) & $rt$ & $r(t-1)+1$ & $f\left( \mathrm{Tr}_{q^n/q^r}(c_0x) \right)-a\mathrm{Tr}_{q^n/q^r}(c_0x)$ & 
$\begin{array}{ccc} 
f(x) \in \mathcal{L}_{r,q}\,\, \text{scattered}\\
\{1,\omega,\ldots,\omega^{t-1}\}\,\,\, \F_{q^r}\text{-basis of }\F_{q^n}\\
c_0=\frac{1}{g'(\omega)}\sum_{j=0}^{t-1}\omega^j a_{i+j+1}\\
g(x)=a_0+a_1x+\ldots+a_{t-1}x^{t-1}+x^t\\ 
g(x)\text{ minimal polynomial of }\omega\text{ over } \F_{q^r}\\
f(x)-ax\,\, \text{is not invertible over}\,\,\F_{q^r}\\
\end{array}$
& \cite{DeBoeckVdV2016,GW2003,NPSZ202x}\\ \hline
\end{tabular}
\caption{Known examples of $\mathcal{D}_{f}$ which are $i$-club}
\label{clubpoly}
\end{table}

\newpage

\begin{table}[htp]
\tabcolsep=0.2 mm
\begin{tabular}{|c|c|c|c|c|c|}
\hline
\hspace{0.5cm} & \hspace{0.2cm}$n$\hspace{0.2cm} & $f(x)$ & \mbox{conditions} & \mbox{references} \\ \hline
i) & $2t$ & $\mathrm{Tr}_{q^n/q^t}\left( f\left(\frac{x}{\epsilon^{q^t}-\epsilon} \right) \right)+\mathrm{Tr}_{q^n/q^t}\left( \frac{\epsilon^{q^t} x}{\epsilon^{q^t}-\epsilon}\right)$ & $\begin{array}{ccc}
\{1,\epsilon\}\,\,\F_{q^t}\text{-basis of }\fqn\\
f(z)=\sum_{i=0}^{t-1} A_iz^{q^i} \in \mathcal{L}_{t,q} \text{ scattered}\\
\end{array}$ & \cite{NPSZ2021} \\\hline
ii) & $2t$ & $\sum_{k=0}^{n-1}\left( \sum_{\ell=0}^{t-1} (u_\ell+u_\ell^{q^s}\xi ){\lambda_{\ell}^*}^{q^k} \right)x^{q^k}$ & $\begin{array}{ccc}
\gcd(s,t)=1\\
\{1,\xi\}\,\,\F_{q^t}\text{-basis of }\fqn\\
\mu \in \F_{q^t} \colon \N_{q^t/q}(\mu)\neq 1 \\ \N_{q^{t}/q}(-\xi^{q^t+1}\mu)\neq (-1)^t\\
\{u_0,\ldots,u_{t-1}\}\,\, \fq\text{-basis of }\F_{q^t}\\
\mathcal{B}=(u_0+\mu u_0^{q^s}\xi,\ldots,u_{t-1}+ \mu u_{t-1}^{q^s}\xi,\\u_0+u_0^{q^s}\xi,\ldots,u_{t-1}+u_{t-1}^{q^s}\xi)\,\,\fq\text{-basis } \fqn\\
(\lambda_0^*,\ldots,\lambda_{n-1}^*)\text{ dual basis of }\fqn
\end{array}$ & \cite{NPSZ2021} \\\hline
\end{tabular}
\caption{Known examples of  $\mathcal{D}_{f}$ having exactly two points of weight $n/2$ and all the others of weight one}
\label{2wpoly}
\end{table}

\begin{remark}
A detailed list of the known examples of $i$-club is provided in \cite{DeBoeckVdV2016}, whereas a polynomial description for the ii)-iv) has been obtained in \cite{NPSZ202x}.
\end{remark}

\begin{remark}
The examples presented in Table \ref{2wpoly}, were originally obtained in the case $n=4$ in the paper \cite{BoPol}.
\end{remark}

\end{document}